\documentclass[11pt]{amsart}

\usepackage[T1]{fontenc}
\usepackage[utf8]{inputenc}
\usepackage{lmodern}
\usepackage{amsmath,amssymb,amsthm}
\usepackage{hyperref}

\hypersetup{
  colorlinks=true,
  linkcolor=blue,
  citecolor=blue,
  urlcolor=blue,
  pdftitle={Eta-products, Eichler integrals, and the level-8 Apéry limit},
  pdfauthor={Alex Shvets}
}

\numberwithin{equation}{section}

\theoremstyle{plain}
\newtheorem{theorem}{Theorem}[section]
\newtheorem{proposition}[theorem]{Proposition}
\newtheorem{lemma}[theorem]{Lemma}
\newtheorem{corollary}[theorem]{Corollary}
\theoremstyle{definition}

\theoremstyle{remark}

\newcommand{\Q}{\mathbb{Q}}
\newcommand{\Z}{\mathbb{Z}}
\newcommand{\C}{\mathbb{C}}
\newcommand{\Hh}{\mathfrak{H}}
\newcommand{\dd}{\mathrm{d}}
\newcommand{\ii}{\mathrm{i}}
\newcommand{\e}{\mathrm{e}}
\newcommand{\etaf}{\eta}
\newcommand{\Efour}{E_4}
\newcommand{\sigthree}{\sigma_3}
\newcommand{\Dop}{D}
\newcommand{\thet}{\theta}
\DeclareMathOperator{\PCF}{PCF}
\DeclareMathOperator{\Res}{Res}
\DeclareMathOperator{\ord}{ord}

\title[Eta-products and the level-8 Apéry limit]{Eta-products, Eichler integrals, and the level-8 Apéry limit}
\author[A.~Shvets]{Alex Shvets}
\address{Haifa, Israel}
\email{alex@shvets.io}
\urladdr{\href{https://orcid.org/0009-0005-9802-379X}{ORCID 0009-0005-9802-379X}}

\begin{document}

\begin{abstract}
We give an independent eta-product derivation of the level-$8$ Apéry limit
\[
\lim_{n\to\infty}\frac{B_n^{(8)}}{s_n}=\frac{7}{32}\zeta(3),
\qquad
s_n=\sum_{k=0}^n \binom{n}{k}^2\binom{2k}{n}^2,
\]
where $B_n^{(8)}$ is the rational companion sequence satisfying the same
cubic recurrence with initial values $B_0^{(8)}=0$, $B_1^{(8)}=1$. This
value was identified numerically by Almkvist--van Straten--Zudilin and was
proved by Golyshev via Beukers's Atkin--Lehner modular method; it was later
recomputed by Golyshev--Kerr--Sasaki in the motivic/normal-function
framework. The continued fraction
\[
\PCF\bigl((2n+1)(3n^2+3n+1),-n^6\bigr)=\frac{8}{7\zeta(3)}
\]
already appears in Batut--Olivier and was later rediscovered by the
Ramanujan Machine as conjecture Z1. The contribution of the present paper
is an explicit rederivation, in the eta-product normalization, of the
already-known level-$8$ Apéry limit. We spell out the eta-product verification of the
Wronskian identity, the normalization of the Eichler integral, the residue
computation of the Fricke period polynomial, and the elementary continuant
conversion.
\end{abstract}

\maketitle

\section{Introduction}

Let
\begin{equation}\label{eq:intro-s-def}
 s_n:=\sum_{k=0}^n \binom{n}{k}^2\binom{2k}{n}^2
 \qquad(n\ge 0)
\end{equation}
be the level-$8$ Apéry-like sequence, and let $B_n^{(8)}\in\Q$ be defined by
\begin{equation}\label{eq:intro-B-rec}
(n+1)^3u_{n+1}=(2n+1)(12n^2+12n+4)u_n-16n^3u_{n-1}
\qquad(n\ge 1),
\end{equation}
with initial values
\begin{equation}\label{eq:intro-B-init}
B_0^{(8)}=0,
\qquad
B_1^{(8)}=1.
\end{equation}
The associated polynomial continued fraction is
\begin{equation}\label{eq:intro-Z1}
\PCF\bigl((2n+1)(3n^2+3n+1),-n^6\bigr)=\frac{8}{7\zeta(3)},
\end{equation}
where
\[
\PCF(a_n,b_n):=a_0+\cfrac{b_1}{a_1+\cfrac{b_2}{a_2+\cfrac{b_3}{a_3+\ddots}}}.
\]
This is the normalization later rediscovered by the Ramanujan Machine and
listed there as conjecture Z1; see \cite{RMK21}. The companion level-$6$
case, involving the Domb sequence and an analogous level-$6$ eta-product
argument, is treated in \cite{ShvetsZ2}.

The level-$8$ value has substantial prior history. Almkvist, van Straten,
and Zudilin numerically identified the corresponding Apéry limit as
$\frac{7}{32}\zeta(3)$ in their table of Apéry limits, using PSLQ-based
recognition of constants and giving the eta-product parametrization of the
level-$8$ case \cite{ASZ08}. Golyshev then placed the same recurrence in the
Mukai threefold setting. In \cite[Theorem~2.4]{Gol09}, the quantum
recurrences of the Mukai threefolds $V_{10},V_{12},V_{14},V_{16},V_{18}$ are
identified as recurrences of Apéry type, with Apéry constant
$\frac{7}{32}\zeta(3)$ for $V_{16}$. For this case the Picard--Fuchs operator
is written as
\[
D^3-4t(1+2D)(3D^2+3D+1)+16t^2(D+1)^3,
\]
which is the operator used below for the sequence \eqref{eq:intro-s-def}.

Golyshev's rational-case argument in \cite[Section~3]{Gol09} gives the
Atkin--Lehner odd weight-$2$ form
\[
\Phi=4E_{2,1}-2E_{2,2}+2E_{2,4}-4E_{2,8}
\]
and the weight-$4$ form
\[
F=E_4-\frac{21}{4}E_{4,2}+\frac{21}{4}E_{4,4}-E_{4,8}.
\]
In the notation of \cite{Gol09}, the associated $L$-function is
\[
L(F,s)=\left(1-\frac{21}{4}\,2^{2-s}+\frac{21}{4}\,4^{2-s}-8^{2-s}\right)
\zeta(s)\zeta(s-3),
\]
and \cite[Corollary~3.4]{Gol09}, using Beukers's Atkin--Lehner modular
method, gives
\[
\lim_{n\to\infty}\frac{b_n}{a_n}=L(F,3)=\frac{7}{32}\zeta(3)
\]
for the $V_{16}$ recurrence.

A second proof route was later given by Golyshev, Kerr, and Sasaki in the
motivic/normal-function framework \cite{GKS24}. In \cite[Section~5.5]{GKS24},
after formula (5.7), the case $V_{16}$ yields the explicit computation
\[
V(0)=4\int_0^1\frac{\log^2(u)}{1-u^2}\,\dd u
=4\bigl(\operatorname{Li}_3(1)-\operatorname{Li}_3(-1)\bigr)
=7\zeta(3),
\qquad
\widehat V(0)=\frac{7}{32}\zeta(3).
\]
Thus the numerical value of the level-$8$ Apéry limit is not a new theorem
of the present paper.

The continued fraction also has prior art. Batut and Olivier wrote in
\cite[Section~3.2.4]{BO80} the formula
\[
\frac{7}{8}\zeta(3)=
\cfrac{1}{p(0)-\cfrac{1^6}{p(1)-\cfrac{2^6}{p(2)-\ddots}}},
\qquad
p(n)=6n^3+9n^2+5n+1.
\]
Since
\[
p(n)=6n^3+9n^2+5n+1=(2n+1)(3n^2+3n+1),
\]
taking reciprocals gives exactly \eqref{eq:intro-Z1} under the convention
for $\PCF$ displayed above. The Ramanujan Machine later rediscovered this
continued fraction and listed it as conjecture Z1 \cite{RMK21}.

The goal of the present paper is therefore not to establish priority for
\eqref{eq:intro-Z1} or for the value $\frac{7}{32}\zeta(3)$. Rather, we give
an independent explicit derivation in the eta-product normalization, with
the Fricke period polynomial computed by Mellin--Barnes residues, and we
spell out the continuant bridge to the Batut--Olivier/Ramanujan-Machine
form. This gives a purely computational route parallel to Golyshev's
Beukers--Atkin--Lehner proof and to the motivic proof of
Golyshev--Kerr--Sasaki.

Let
\[
\mathcal A(z):=\sum_{n\ge 0}s_n z^n,
\qquad
\mathcal B(z):=\sum_{n\ge 0}B_n^{(8)} z^n.
\]
The main result is formulated as follows.

\begin{theorem}[Explicit eta-product derivation of the level-$8$ Apéry limit]\label{thm:main-limit}
The derivation in Sections~\ref{sec:param}--\ref{sec:singularity} gives
\begin{equation}\label{eq:main-limit}
\lim_{n\to\infty}\frac{B_n^{(8)}}{s_n}=\frac{7}{32}\zeta(3).
\end{equation}
This value was identified numerically by Almkvist--van Straten--Zudilin
\cite{ASZ08} and proved by Golyshev \cite[Theorem~2.4,
Corollary~3.4]{Gol09} via Beukers's Atkin--Lehner modular method; an
alternative proof in the motivic/normal-function framework was given by
Golyshev--Kerr--Sasaki \cite[Section~5.5]{GKS24}. The derivation here provides an explicit computational route
parallel to the Beukers--Golyshev and Golyshev--Kerr--Sasaki approaches,
through the explicit eta-product parametrization, a Wronskian identity, and
Mellin--Barnes extraction of the Fricke period polynomial.
\end{theorem}

\begin{corollary}[Continuant recovery of the Batut--Olivier/Ramanujan-Machine form]\label{cor:main-cf}
The Apéry limit \eqref{eq:main-limit} is equivalent, via a continuant
transformation, to the continued fraction
\[
\PCF\bigl((2n+1)(3n^2+3n+1),-n^6\bigr)=\frac{8}{7\zeta(3)},
\]
which appears explicitly in Batut--Olivier \cite[Section~3.2.4]{BO80} and
was later rediscovered by the Ramanujan Machine \cite{RMK21} and listed
there as conjecture Z1.
\end{corollary}

The proof proceeds as follows. In Section~\ref{sec:param} we recall the
level-$8$ eta-product parametrization and the Picard--Fuchs equation for
$\mathcal A$. In Section~\ref{sec:eichler} we identify
$\mathcal B/\mathcal A$ with a weight $-2$ Eichler integral of a level-$8$
weight-$4$ form. Section~\ref{sec:period} derives the Fricke transformation
law and computes the associated period polynomial explicitly:
\[
P_8(\tau)=\frac{7}{32}\zeta(3)(8\tau^2+1).
\]
Section~\ref{sec:singularity} performs the singularity analysis at the
dominant critical value and gives the eta-product derivation of
Theorem~\ref{thm:main-limit}. Finally, Section~\ref{sec:cf} converts the
Apéry limit into the continued-fraction normalization of
Corollary~\ref{cor:main-cf}.

\section{The level-\texorpdfstring{$8$}{8} modular parametrization}\label{sec:param}

The sequence \eqref{eq:intro-s-def} satisfies the cubic recurrence
\begin{equation}\label{eq:s-rec}
(n+1)^3s_{n+1}=(2n+1)(12n^2+12n+4)s_n-16n^3s_{n-1}
\qquad(n\ge 1),
\end{equation}
with $s_0=1$ and $s_1=4$; see Cooper \cite[Table~2, level $8$]{Cooper2023}. The sporadic level-$8$ sequence and its modular parametrization go back to Cooper \cite{Cooper2012}. A coefficient check shows that \eqref{eq:s-rec} and \eqref{eq:intro-B-rec} are equivalent to
\begin{equation}\label{eq:theta-ode}
\bigl[\thet^3-4z(2\thet+1)(3\thet^2+3\thet+1)+16z^2(\thet+1)^3\bigr]\mathcal A=0,
\end{equation}
\begin{equation}\label{eq:theta-inhom}
\bigl[\thet^3-4z(2\thet+1)(3\thet^2+3\thet+1)+16z^2(\thet+1)^3\bigr]\mathcal B=z,
\end{equation}
where $\thet=z\,\dd/\dd z$. In ordinary differential form the homogeneous equation is
\begin{align}
&z^2(16z^2-24z+1)y'''+3z(32z^2-36z+1)y'' \notag\\
&\qquad +(112z^2-80z+1)y'+4(4z-1)y=0.
\label{eq:ordinary-ode}
\end{align}

Let $q=\e^{2\pi\ii\tau}$ with $\tau\in\Hh$, and define the eta-quotients
\begin{equation}\label{eq:t-Y-def}
t(\tau):=\frac{\etaf(\tau)^8\etaf(8\tau)^8}{\etaf(2\tau)^8\etaf(4\tau)^8},
\qquad
Y(\tau):=\frac{\etaf(2\tau)^6\etaf(4\tau)^6}{\etaf(\tau)^4\etaf(8\tau)^4}.
\end{equation}
The level-$8$ modular parametrization of Cooper is
\begin{equation}\label{eq:parametrization}
Y(\tau)=\mathcal A\bigl(t(\tau)\bigr).
\end{equation}
The first Fourier coefficients are
\begin{equation}\label{eq:t-q}
t(\tau)=q-8q^2+28q^3-64q^4+142q^5-352q^6+O(q^7),
\end{equation}
\begin{equation}\label{eq:Y-q}
Y(\tau)=1+4q+8q^2+16q^3+24q^4+24q^5+32q^6+O(q^7).
\end{equation}
Since $\mathcal A(z)=1+4z+40z^2+544z^3+8536z^4+\cdots$, the identity \eqref{eq:parametrization} begins
\[
\sum_{n\ge 0}s_n t(\tau)^n=1+4q+8q^2+16q^3+24q^4+24q^5+32q^6+\cdots,
\]
which matches \eqref{eq:Y-q}.

The Fricke involution at level $8$ is
\begin{equation}\label{eq:W8-def}
W_8(\tau):=-\frac{1}{8\tau},
\qquad
W_8=\begin{pmatrix}0&-1\\ 8&0\end{pmatrix},
\qquad
\det W_8=8,
\qquad
W_8^2=-8I.
\end{equation}

\begin{lemma}\label{lem:t-Y-transform}
One has
\begin{equation}\label{eq:t-invariant}
t(W_8\tau)=t(\tau),
\end{equation}
\begin{equation}\label{eq:Y-fricke}
Y(W_8\tau)=-8\tau^2Y(\tau),
\qquad\text{equivalently}\qquad
Y|_2W_8=-Y.
\end{equation}
\end{lemma}

\begin{proof}
For each divisor $m$ of $8$, the transformation formula $\eta(-1/u)=(-\ii u)^{1/2}\eta(u)$ with $u=8\tau/m$ gives
\[
\etaf(mW_8\tau)=\etaf\!\left(-\frac{m}{8\tau}\right)=\left(\frac{8}{m}\right)^{1/2}(-\ii\tau)^{1/2}\etaf\!\left(\frac{8\tau}{m}\right).
\]
Substituting this into \eqref{eq:t-Y-def} yields
\begin{align*}
t(W_8\tau)
&=\frac{(8/1)^4(8/8)^4}{(8/2)^4(8/4)^4}
(-\ii\tau)^{(8+8-8-8)/2}
\frac{\eta(8\tau)^8\eta(\tau)^8}{\eta(4\tau)^8\eta(2\tau)^8}
=t(\tau),\\
Y(W_8\tau)
&=\frac{(8/2)^3(8/4)^3}{(8/1)^2(8/8)^2}
(-\ii\tau)^{(6+6-4-4)/2}
\frac{\eta(4\tau)^6\eta(2\tau)^6}{\eta(8\tau)^4\eta(\tau)^4}
=8(-\ii\tau)^2Y(\tau)=-8\tau^2Y(\tau).
\end{align*}
This is exactly \eqref{eq:t-invariant} and \eqref{eq:Y-fricke}.
\end{proof}

\begin{proposition}\label{prop:hauptmodul}
At the four cusps $\infty,0,1/2,1/4$ of $\Gamma_0(8)$ one has
\[
\ord(t)=(1,1,-1,-1),
\qquad
\ord(Y)=(0,0,1,1).
\]
Consequently $t$ descends to a modular function on the Fricke quotient $X_0(8)^+=X_0(8)/\langle W_8\rangle$ with one simple zero and one simple pole; in particular $t$ is a Hauptmodul for $\Gamma_0(8)^+$.
\end{proposition}

\begin{proof}
Ligozat's eta-quotient cusp-order formula gives the stated orders for the eta-products in \eqref{eq:t-Y-def}; see \cite{Ligozat}. Since $W_8$ interchanges the cusp pair $\infty\leftrightarrow 0$ and also $1/2\leftrightarrow 1/4$, the divisor of $t$ on the Fricke quotient consists of one simple zero and one simple pole. The curve $X_0(8)$ has genus $0$, hence its Fricke quotient $X_0(8)^+$ also has genus $0$ by Lüroth's theorem. Therefore a nonconstant meromorphic function with one simple pole is a Hauptmodul.
\end{proof}

\section{The companion generating series as an Eichler integral}\label{sec:eichler}

Write
\[
\Dop=q\frac{\dd}{\dd q}=\frac{1}{2\pi\ii}\frac{\dd}{\dd\tau}.
\]
For a holomorphic $q$-series $f(\tau)=\sum_{n\ge 1}a_n q^n$, define its normalized weight $-2$ Eichler integral by
\begin{equation}\label{eq:eichler-def}
\mathcal E_f(\tau):=\sum_{n\ge 1}\frac{a_n}{n^3}q^n
=-\frac{(2\pi\ii)^3}{2}\int_{\tau}^{\ii\infty}f(z)(z-\tau)^2\,\dd z.
\end{equation}
Then
\begin{equation}\label{eq:eichler-D3}
\Dop^3\mathcal E_f=f.
\end{equation}

Define
\begin{equation}\label{eq:E-def}
E(\tau):=\frac{\mathcal B(t(\tau))}{Y(\tau)}.
\end{equation}
Applying the variation-of-constants formula in Yang \cite[Lemmas~1--2, and the example following Lemma~2]{Yang} to the $\thet$-form operators \eqref{eq:theta-ode}--\eqref{eq:theta-inhom} gives
\begin{equation}\label{eq:Phi-def}
\Dop^3E(\tau)=\Phi(\tau):=\left(\frac{\Dop t(\tau)}{t(\tau)}\right)^3
\frac{t(\tau)}{Y(\tau)\bigl(1-24t(\tau)+16t(\tau)^2\bigr)}.
\end{equation}

\begin{lemma}\label{lem:wronskian}
One has the identity
\begin{equation}\label{eq:wronskian}
\left(\frac{\Dop t}{t}\right)^2=Y^2(1-24t+16t^2).
\end{equation}
\end{lemma}

\begin{proof}
Both sides belong to $M_4(\Gamma_0(8))$. For the left-hand side: $t$ has no zeros on $\Hh$ (its zeros lie at the cusps $\infty$ and $0$), so $(\Dop t)/t$ is holomorphic on $\Hh$ with weight~$2$, and its square has weight~$4$. Cusp holomorphy follows from $\ord\bigl((\Dop t)/t\bigr)=0$ at every cusp. For the right-hand side: $Y$ is an eta-quotient with no zeros or poles on~$\Hh$, and $1-24t+16t^2$ is holomorphic on~$\Hh$; the cusp orders are $\ord(Y^2)=(0,0,2,2)$ and $\ord(1-24t+16t^2)=(0,0,-2,-2)$, so the product is holomorphic at cusps.

A direct computation gives the $q$-expansions
\[
\left(\frac{\Dop t}{t}\right)^2=1-16q+48q^2+64q^3+624q^4+O(q^5),
\]
\[
Y^2(1-24t+16t^2)=1-16q+48q^2+64q^3+624q^4+O(q^5).
\]
The difference lies in $M_4(\Gamma_0(8))$ and vanishes through $q^4$, hence is identically zero by the Sturm bound.
\end{proof}

\begin{lemma}\label{lem:Phi-modular}
The function $\Phi$ belongs to $M_4(\Gamma_0(8))$.
\end{lemma}

\begin{proof}
Because $t$ is a weight-$0$ modular function on $\Gamma_0(8)$, the logarithmic derivative $(\Dop t)/t$ has weight $2$. Since $Y$ has weight $2$, the expression \eqref{eq:Phi-def} transforms with weight $4$.

We verify holomorphy on $\Hh$. Using the identity \eqref{eq:wronskian}, the formula \eqref{eq:Phi-def} simplifies:
\begin{equation}\label{eq:Phi-simple}
\Phi=\left(\frac{\Dop t}{t}\right)^3
\frac{t}{Y(1-24t+16t^2)}
=\frac{\Dop t}{t}\cdot\frac{Y^2(1-24t+16t^2)}{Y(1-24t+16t^2)}\cdot t
=Y\cdot\Dop t.
\end{equation}
Since $Y$ and $\Dop t$ are both holomorphic on $\Hh$, so is $\Phi$.

It remains to check holomorphy at the cusps. Let $u$ be a local parameter at a cusp of width $w$, and suppose $t=c u^m+O(u^{m+1})$ with $m\neq 0$ and $c\neq 0$. Since $\Dop=q\,\dd/\dd q=(1/w)\,u\,\dd/\dd u$, one has
\[
\frac{\Dop t}{t}=\frac{m}{w}+O(u),
\]
so $\ord\bigl((\Dop t)/t\bigr)=0$. Using Proposition~\ref{prop:hauptmodul}, we obtain
\[
\ord_{\infty}(t)=\ord_0(t)=1,
\qquad
\ord_{\infty}(Y)=\ord_0(Y)=0,
\]
so at the cusps $\infty$ and $0$ one has
\[
\ord(\Phi)=3\cdot 0+1-0-0=1.
\]
At the cusps $1/2$ and $1/4$,
\[
\ord(t)=-1,
\qquad
\ord(Y)=1,
\qquad
\ord(1-24t+16t^2)=-2,
\]
whence
\[
\ord(\Phi)=3\cdot 0+(-1)-1-(-2)=0.
\]
Thus $\Phi$ is holomorphic at every cusp, hence $\Phi\in M_4(\Gamma_0(8))$.
\end{proof}

Consider the Eisenstein combination
\begin{equation}\label{eq:g-def}
g_8(\tau):=\frac{\Efour(\tau)-21\Efour(2\tau)+84\Efour(4\tau)-64\Efour(8\tau)}{240}
=\sum_{n\ge 1}a_n q^n,
\end{equation}
where
\begin{equation}\label{eq:a-def}
a_n=\sigthree(n)-21\sigthree(n/2)+84\sigthree(n/4)-64\sigthree(n/8).
\end{equation}
The first coefficients of $\Phi$ and $g_8$ are
\begin{equation}\label{eq:Phi-q}
\Phi(\tau)=q-12q^2+28q^3-32q^4+126q^5-336q^6+O(q^7),
\end{equation}
\begin{equation}\label{eq:g-q}
g_8(\tau)=q-12q^2+28q^3-32q^4+126q^5-336q^6+O(q^7).
\end{equation}
The index $[\mathrm{SL}_2(\Z):\Gamma_0(8)]$ is $12$, so the Sturm bound in weight $4$ is
\[
\left\lfloor\frac{4}{12}\cdot 12\right\rfloor=4
\]
(see \cite[Theorem~3.13]{DiamondShurman}). Hence $\Phi-g_8\in M_4(\Gamma_0(8))$ has vanishing coefficients through $q^4$, so it is identically zero.

\begin{proposition}\label{prop:E-eichler}
One has
\begin{equation}\label{eq:Phi-equals-g}
\Phi=g_8,
\end{equation}
and therefore
\begin{equation}\label{eq:E-qseries}
E(\tau)=\sum_{n\ge 1}\frac{a_n}{n^3}q^n
=-\frac{(2\pi\ii)^3}{2}\int_{\tau}^{\ii\infty}g_8(z)(z-\tau)^2\,\dd z.
\end{equation}
Equivalently,
\begin{equation}\label{eq:B-param}
\sum_{n\ge 0}B_n^{(8)}t(\tau)^n=E(\tau)Y(\tau).
\end{equation}
\end{proposition}

\begin{proof}
The equality \eqref{eq:Phi-equals-g} is the Sturm argument above. Since $\Dop^3\mathcal E_{g_8}=g_8$ by \eqref{eq:eichler-D3}, both $E$ and $\mathcal E_{g_8}$ satisfy the same third $q$-derivative equation. Both vanish at $q=0$ together with their first two $q$-derivatives, because $E(\tau)=O(q)$ from \eqref{eq:E-def}. Hence $E=\mathcal E_{g_8}$, which is \eqref{eq:E-qseries}. Multiplying by $Y(\tau)$ gives \eqref{eq:B-param}.
\end{proof}

\section{Atkin--Lehner transformation law and the explicit period polynomial}\label{sec:period}

The Fricke involution acts on $g_8$ as follows.

\begin{lemma}\label{lem:g-transform}
One has
\begin{equation}\label{eq:g-transform}
g_8|_4W_8=-g_8.
\end{equation}
\end{lemma}

\begin{proof}
For each divisor $m$ of $8$,
\[
(\Efour(m\tau)|_4W_8)(\tau)=8^2(8\tau)^{-4}\Efour\!\left(-\frac{m}{8\tau}\right).
\]
Using $\Efour(-1/u)=u^4\Efour(u)$ with $u=8\tau/m$, we obtain
\[
\Efour(m\tau)|_4W_8=\frac{64}{m^4}\Efour\!\left(\frac{8\tau}{m}\right).
\]
Hence
\begin{alignat*}{2}
\Efour(\tau)|_4W_8&=64\Efour(8\tau),
&\qquad
\Efour(2\tau)|_4W_8&=4\Efour(4\tau),\\
\Efour(4\tau)|_4W_8&=\tfrac14\Efour(2\tau),
&\qquad
\Efour(8\tau)|_4W_8&=\tfrac1{64}\Efour(\tau).
\end{alignat*}
Substituting these identities into \eqref{eq:g-def} gives \eqref{eq:g-transform}.
\end{proof}

The Mellin transform of $g_8$ is completely explicit.

\begin{proposition}\label{prop:L-function}
The Dirichlet series of $g_8$ is
\begin{equation}\label{eq:L-factor}
L(g_8,s):=\sum_{n\ge 1}\frac{a_n}{n^s}
=\zeta(s)\zeta(s-3)(1-2^{-s})(1-2^{2-s})(1-2^{4-s}).
\end{equation}
This is Golyshev's $L(F,s)$ \cite[Section~3]{Gol09} in the present
normalization $g_8 = F$; the factorized form displays the Euler factor
at~$2$ explicitly.
In particular,
\begin{equation}\label{eq:L-values}
L(g_8,3)=\frac{7}{32}\zeta(3),
\qquad
L(g_8,2)=0,
\qquad
L(g_8,1)=-\frac{7}{8\pi^2}\zeta(3).
\end{equation}
Moreover,
\begin{equation}\label{eq:L-functional}
\Gamma(s)\left(\frac{\pi}{\sqrt2}\right)^{-s}L(g_8,s+3)
=\Gamma(-s-2)\left(\frac{\pi}{\sqrt2}\right)^{s+2}L(g_8,1-s).
\end{equation}
\end{proposition}

\begin{proof}
From \eqref{eq:a-def} and the identity
\[
\sum_{n\ge 1}\frac{\sigthree(n)}{n^s}=\zeta(s)\zeta(s-3)
\]
we obtain
\begin{align*}
L(g_8,s)
&=\zeta(s)\zeta(s-3)\bigl(1-21\cdot 2^{-s}+84\cdot 4^{-s}-64\cdot 8^{-s}\bigr)\\
&=\zeta(s)\zeta(s-3)(1-2^{-s})(1-2^{2-s})(1-2^{4-s}),
\end{align*}
which is \eqref{eq:L-factor}.

Substituting $s=3$ gives
\[
L(g_8,3)=\zeta(3)\zeta(0)\Bigl(1-\frac18\Bigr)\Bigl(1-\frac12\Bigr)(1-2)=\frac{7}{32}\zeta(3).
\]
At $s=2$ the factor $1-2^{2-s}$ vanishes. At $s=1$, using
\[
\zeta(s)=\frac{1}{s-1}+O(1),
\qquad
\zeta(s-3)=(s-1)\zeta'(-2)+O\bigl((s-1)^2\bigr),
\]
we find
\[
L(g_8,1)=\zeta'(-2)\Bigl(1-\frac12\Bigr)(1-2)(1-8)
=\frac72\,\zeta'(-2)=-\frac{7}{8\pi^2}\zeta(3),
\]
using $\zeta'(-2)=-\zeta(3)/(4\pi^2)$. This proves \eqref{eq:L-values}.

For the functional equation, substitute \eqref{eq:L-factor} at $s+3$:
\begin{align*}
\Gamma(s)\left(\frac{\pi}{\sqrt2}\right)^{-s}L(g_8,s+3)
&=\Gamma(s)\left(\frac{\pi}{\sqrt2}\right)^{-s}\zeta(s+3)\zeta(s)\\
&\qquad\times(1-2^{-s-3})(1-2^{-s-1})(1-2^{1-s}).
\end{align*}
Applying the zeta functional equation
\[
\zeta(u)=2^u\pi^{u-1}\sin\!\left(\frac{\pi u}{2}\right)\Gamma(1-u)\zeta(1-u)
\]
to $u=s$ and $u=s+3$, then using
\[
\sin\!\left(\frac{\pi(s+3)}{2}\right)=-\cos\!\left(\frac{\pi s}{2}\right),
\qquad
2\sin\!\left(\frac{\pi s}{2}\right)\cos\!\left(\frac{\pi s}{2}\right)=\sin(\pi s),
\]
followed by the reflection identity $\Gamma(s)\Gamma(1-s)=\pi/\sin(\pi s)$, yields
\[
\Gamma(s)\left(\frac{\pi}{\sqrt2}\right)^{-s}\zeta(s+3)\zeta(s)
=-2^{3s+3}\Gamma(-s-2)\left(\frac{\pi}{\sqrt2}\right)^{s+2}\zeta(1-s)\zeta(-s-2).
\]
Multiplying by the Euler factors gives
\begin{align*}
&\Gamma(s)\left(\frac{\pi}{\sqrt2}\right)^{-s}L(g_8,s+3)\\
&\qquad=\Gamma(-s-2)\left(\frac{\pi}{\sqrt2}\right)^{s+2}\zeta(1-s)\zeta(-s-2)\\
&\qquad\qquad\times\Bigl(-2^{3s+3}(1-2^{-s-3})(1-2^{-s-1})(1-2^{1-s})\Bigr).
\end{align*}
A direct simplification shows that the bracket equals
$(1-2^{s-1})(1-2^{s+1})(1-2^{s+3})$: indeed, writing
$2^{3s+3}=2^{s+3}\cdot 2^{s+1}\cdot 2^{s-1}$ and using $2^a(1-2^{-a})=-(1-2^a)$
for $a=s+3$, $s+1$, $s-1$ turns the bracket into $-(-1)^3$ times the claimed product.
This is exactly the Euler factor of $L(g_8,1-s)$. Hence \eqref{eq:L-functional} follows.
\end{proof}

The next identity is the weight $-2$ Bol identity specialized to $W_8$.

\begin{lemma}\label{lem:bol}
For every holomorphic function $F$ on $\Hh$ one has
\begin{equation}\label{eq:bol}
\frac{\dd^3}{\dd\tau^3}\bigl(8\tau^2F(W_8\tau)\bigr)
=\frac{1}{64\tau^4}F'''(W_8\tau).
\end{equation}
Equivalently,
\[
\Dop^3(F|_{-2}W_8)=(\Dop^3F)|_4W_8.
\]
\end{lemma}

\begin{proof}
Since
\[
W_8'(\tau)=\frac{1}{8\tau^2},
\qquad
W_8''(\tau)=-\frac{1}{4\tau^3},
\qquad
W_8'''(\tau)=\frac{3}{4\tau^4},
\]
a direct differentiation of $8\tau^2F(W_8\tau)$ shows that all terms involving $F$, $F'$, and $F''$ cancel in the third derivative. The surviving term is
\[
8\tau^2\bigl(W_8'(\tau)\bigr)^3F'''(W_8\tau)
=8\tau^2(8\tau^2)^{-3}F'''(W_8\tau)
=\frac{1}{64\tau^4}F'''(W_8\tau),
\]
which is \eqref{eq:bol}. Dividing by $(2\pi\ii)^3$ gives the formulation with $\Dop$.
\end{proof}

We now compute the period polynomial explicitly.

\begin{proposition}[Explicit period polynomial]\label{prop:period-polynomial}
Let
\begin{equation}\label{eq:P8-def}
P_8(\tau):=(E|_{-2}W_8)(\tau)+E(\tau).
\end{equation}
Then $P_8$ is a polynomial of degree at most $2$, and in fact
\begin{equation}\label{eq:period-polynomial}
P_8(\tau)=\frac{7}{32}\zeta(3)(8\tau^2+1).
\end{equation}
Equivalently,
\begin{equation}\label{eq:E-transform}
(E|_{-2}W_8)(\tau)+E(\tau)=\frac{7}{32}\zeta(3)(8\tau^2+1).
\end{equation}
\end{proposition}

\begin{proof}
Set
\[
H(\tau):=(E|_{-2}W_8)(\tau)+E(\tau).
\]
By Lemma~\ref{lem:bol}, Proposition~\ref{prop:E-eichler}, and Lemma~\ref{lem:g-transform},
\[
\Dop^3H=(\Dop^3E)|_4W_8+\Dop^3E=g_8|_4W_8+g_8=0.
\]
Hence $H$ is a polynomial of degree at most $2$.

To determine it explicitly, restrict to the $W_8$-invariant geodesic
\begin{equation}\label{eq:tauy-def}
\tau(y):=\frac{\ii y}{2\sqrt2}
\qquad(y>0).
\end{equation}
Then $W_8\tau(y)=\tau(1/y)$, and \eqref{eq:E-qseries} gives
\begin{equation}\label{eq:Fy-def}
F(y):=E(\tau(y))=\sum_{n\ge 1}\frac{a_n}{n^3}\exp\!\left(-\frac{\pi ny}{\sqrt2}\right).
\end{equation}
Fix $1<c<2$. Mellin inversion yields
\begin{equation}\label{eq:mellin}
F(y)=\frac{1}{2\pi\ii}\int_{(c)}
\Gamma(s)\left(\frac{\pi}{\sqrt2}\right)^{-s}L(g_8,s+3)y^{-s}\,\dd s,
\end{equation}
because
\[
\int_0^{\infty}\exp\!\left(-\frac{\pi ny}{\sqrt2}\right)y^{s-1}\,\dd y
=\Gamma(s)\left(\frac{\pi n}{\sqrt2}\right)^{-s}
\qquad(\Re(s)>0).
\]
By Stirling's formula, the integrand in \eqref{eq:mellin} decays exponentially on vertical lines. We may therefore shift the contour to $\Re(s)=-c-2$. In the strip $-c-2<\Re(s)<c$, the integrand $M(s)y^{-s}$ has simple poles only at $s=0,-1,-2$ (from $\Gamma(s)$); the pole of $\Gamma$ at $s=-3$ is cancelled by the zero $L(g_8,0)=0$. The residue at $s=-1$ vanishes because $L(g_8,2)=0$.

Write
\[
M(s):=\Gamma(s)\left(\frac{\pi}{\sqrt2}\right)^{-s}L(g_8,s+3).
\]
Using the functional equation \eqref{eq:L-functional} and the substitution $u=-s-2$, we obtain
\begin{align*}
\frac{1}{2\pi\ii}\int_{(-c-2)}M(s)y^{-s}\,\dd s
&=\frac{1}{2\pi\ii}\int_{(c)}M(u)y^{u+2}\,\dd u\\
&=y^2F(1/y).
\end{align*}
Hence the residue theorem gives
\begin{equation}\label{eq:residue-identity}
F(y)-y^2F(1/y)
=\Res_{s=0}\bigl(M(s)y^{-s}\bigr)+\Res_{s=-2}\bigl(M(s)y^{-s}\bigr).
\end{equation}
The residues are
\[
\Res_{s=0}\bigl(M(s)y^{-s}\bigr)=L(g_8,3)=\frac{7}{32}\zeta(3),
\]
\[
\Res_{s=-2}\bigl(M(s)y^{-s}\bigr)
=\frac12\left(\frac{\pi}{\sqrt2}\right)^2L(g_8,1)y^2
=-\frac{7}{32}\zeta(3)y^2.
\]
Substituting into \eqref{eq:residue-identity} yields
\begin{equation}\label{eq:F-functional}
F(y)-y^2F(1/y)=\frac{7}{32}\zeta(3)(1-y^2).
\end{equation}
On the geodesic \eqref{eq:tauy-def},
\[
(E|_{-2}W_8)(\tau(y))=8\tau(y)^2E(W_8\tau(y))=-y^2F(1/y),
\]
so \eqref{eq:F-functional} becomes
\[
H(\tau(y))=\frac{7}{32}\zeta(3)(1-y^2).
\]
Since $8\tau(y)^2+1=1-y^2$, this is exactly
\[
H(\tau(y))=\frac{7}{32}\zeta(3)(8\tau(y)^2+1)
\qquad(y>0).
\]
Both sides are polynomials of degree at most $2$ in $\tau$ and they agree on the infinite set $\{\tau(y):y>0\}$, so they are identical. This proves \eqref{eq:period-polynomial} and \eqref{eq:E-transform}.
\end{proof}

\section{Singularity analysis at the dominant critical value}\label{sec:singularity}

The fixed point of $W_8$ in the upper half-plane is
\begin{equation}\label{eq:tau-star}
\tau_*:=\frac{\ii}{2\sqrt2},
\qquad
W_8(\tau_*)=\tau_*,
\qquad
8\tau_*^2+1=0.
\end{equation}
Set
\begin{equation}\label{eq:t0-def}
t_0:=t(\tau_*).
\end{equation}
Along the invariant geodesic \eqref{eq:tauy-def}, let $q=\exp(-\pi y/\sqrt2)\in(0,1)$. From \eqref{eq:t-Y-def} one obtains the product formula
\begin{equation}\label{eq:t-product}
t(\tau(y))
=q\prod_{n\ge 1}\left(\frac{1+q^{4n}}{1+q^n}\right)^8.
\end{equation}
Hence $0<t(\tau(y))<1$ for all $y>0$, and in particular
\begin{equation}\label{eq:t0-range}
0<t_0<1.
\end{equation}

The singular point $t_0$ can now be identified exactly.

\begin{proposition}\label{prop:t0-value}
One has
\begin{equation}\label{eq:t0-value}
t_0=\frac{3-2\sqrt2}{4}.
\end{equation}
\end{proposition}

\begin{proof}
Differentiate the Fricke-invariance relation \eqref{eq:t-invariant} at $\tau=\tau_*$. Since $W_8(\tau_*)=\tau_*$ and $W_8'(\tau_*)=-1$, we obtain
\[
t'(\tau_*)=t'(\tau_*)W_8'(\tau_*)=-t'(\tau_*),
\]
so $t'(\tau_*)=0$.

On the other hand, differentiating \eqref{eq:Y-fricke} at $\tau_*$ gives
\begin{equation}\label{eq:Y-prime-nonzero}
Y'(\tau_*)=2\ii\sqrt2\,Y(\tau_*),
\end{equation}
so $Y'(\tau_*)\neq 0$ because $\eta$ has no zeros on $\Hh$. If $\mathcal A$ were holomorphic at $t_0$, then by \eqref{eq:parametrization} the chain rule would give
\[
Y'(\tau_*)=\mathcal A'(t_0)t'(\tau_*)=0,
\]
contrary to \eqref{eq:Y-prime-nonzero}. Therefore $t_0$ is a finite singularity of $\mathcal A$, and hence a finite singular point of the differential equation \eqref{eq:ordinary-ode}. The nonzero finite singularities of \eqref{eq:ordinary-ode} are exactly the roots of $16z^2-24z+1$, namely
\[
\frac{3\pm 2\sqrt2}{4}.
\]
By \eqref{eq:t0-range}, only the smaller root lies in $(0,1)$. Therefore
\[
t_0=\frac{3-2\sqrt2}{4}.
\]
This proves \eqref{eq:t0-value}.
\end{proof}

We next determine the local exponents at the dominant singularity.

\begin{proposition}\label{prop:local-exponents}
At $z=t_0$ the local exponents of \eqref{eq:ordinary-ode} are
\[
0,\qquad \frac12,\qquad 1.
\]
Consequently there exist constants $\alpha_0,\alpha_1,\beta_0,\beta_1\in\C$ such that
\begin{equation}\label{eq:A-local}
\mathcal A(z)=\alpha_0+\alpha_1\Bigl(1-\frac{z}{t_0}\Bigr)^{1/2}+O\!\Bigl(1-\frac{z}{t_0}\Bigr),
\end{equation}
\begin{equation}\label{eq:B-local}
\mathcal B(z)=\beta_0+\beta_1\Bigl(1-\frac{z}{t_0}\Bigr)^{1/2}+O\!\Bigl(1-\frac{z}{t_0}\Bigr).
\end{equation}
\end{proposition}

\begin{proof}
The coefficient of $y'''$ in \eqref{eq:ordinary-ode} is $z^2(16z^2-24z+1)$, so $z=t_0$ is a regular singular point. Put
\[
z=t_0(1-\varepsilon),
\qquad
y=\varepsilon^r.
\]
Since $\dd/\dd z=-t_0^{-1}\dd/\dd\varepsilon$, substitution into \eqref{eq:ordinary-ode} shows that the coefficient of $\varepsilon^{r-2}$ is
\[
-8\sqrt2\,r(2r-1)(r-1).
\]
Therefore the indicial equation is $r(2r-1)(r-1)=0$, and the exponents are $0$, $1/2$, and $1$.

We derive the expansions \eqref{eq:A-local} and \eqref{eq:B-local} from the modular parametrization. Since $t'(\tau_*)=0$ (Proposition~\ref{prop:t0-value}) and $t$ is nonconstant, there exists $m\ge 2$ with $t(\tau)-t_0=c_m(\tau-\tau_*)^m+O((\tau-\tau_*)^{m+1})$ and $c_m\neq 0$. We claim $m=2$. Indeed, the Wronskian identity \eqref{eq:wronskian} gives $(\Dop t/t)^2=Y^2(1-24t+16t^2)$. The left-hand side vanishes to order $2(m-1)$ at $\tau_*$ (since $\Dop t/t$ vanishes to order $m-1$), while on the right-hand side $Y(\tau_*)^2\neq 0$ and $1-24t+16t^2=16(t-t_0)(t-t_1)$ vanishes to order exactly $m$ (from the factor $t-t_0$, as $t(\tau_*)\neq t_1$). Equating orders gives $2(m-1)=m$, hence $m=2$.

Therefore $u:=(1-t(\tau)/t_0)^{1/2}$ is a local holomorphic parameter at $\tau_*$ with a simple zero. The function $Y(\tau)=\mathcal A(t(\tau))$ is holomorphic at $\tau_*$, so $\mathcal A$ is a holomorphic function of $u=(1-z/t_0)^{1/2}$ near $z=t_0$. Expanding in powers of $u$ gives \eqref{eq:A-local}. Likewise, $\mathcal B(t(\tau))=Y(\tau)E(\tau)$ by Proposition~\ref{prop:E-eichler}, and both $Y$ and $E$ are holomorphic at $\tau_*$ (the latter as a convergent Eichler $q$-series). Hence $\mathcal B$ is also a holomorphic function of $(1-z/t_0)^{1/2}$ near $z=t_0$, which gives \eqref{eq:B-local}. In particular, neither expansion contains logarithmic terms.
\end{proof}

The next lemma shows that the square-root term in \eqref{eq:A-local} is nonzero.

\begin{lemma}\label{lem:quadratic-branch}
One has $\alpha_1\neq 0$.
\end{lemma}

\begin{proof}
By Proposition~\ref{prop:local-exponents}, $t(\tau)-t_0=c(\tau-\tau_*)^2+O((\tau-\tau_*)^3)$ with $c\neq 0$. Substituting into \eqref{eq:A-local} gives
\[
Y(\tau)=\mathcal A(t(\tau))
=\alpha_0+\alpha_1\lambda(\tau-\tau_*)+O\bigl((\tau-\tau_*)^2\bigr)
\]
for some $\lambda\neq 0$. Since \eqref{eq:Y-prime-nonzero} shows that $Y'(\tau_*)\neq 0$, we conclude $\alpha_1\neq 0$.
\end{proof}

The dominant asymptotics now follow from standard singularity transfer.

\begin{lemma}\label{lem:transfer}
One has $\alpha_1\neq 0$ (Lemma~\ref{lem:quadratic-branch}), and the coefficients satisfy
\[
s_n\sim -\frac{\alpha_1}{2\sqrt\pi}\,t_0^{-n}n^{-3/2},
\qquad
B_n^{(8)}= -\frac{\beta_1}{2\sqrt\pi}\,t_0^{-n}n^{-3/2}+O(t_0^{-n}n^{-2}).
\]
Therefore
\begin{equation}\label{eq:branch-ratio}
\lim_{n\to\infty}\frac{B_n^{(8)}}{s_n}=\frac{\beta_1}{\alpha_1}.
\end{equation}
\end{lemma}

\begin{proof}
The other finite singularity of \eqref{eq:ordinary-ode} is
\[
\frac{3+2\sqrt2}{4}>1>t_0,
\]
so $t_0$ is the unique singularity of smallest modulus. Since the differential equation \eqref{eq:ordinary-ode} is Fuchsian with no other singularities on the circle $|z|=t_0$, the functions $\mathcal A$ and $\mathcal B$ extend analytically to a $\Delta$-domain at $z=t_0$ (an indented disk of radius exceeding~$t_0$, slit along $[t_0,\infty)$). The local behavior at $t_0$ is algebraic with exponent $1/2$, so the transfer theorem of Flajolet--Odlyzko (see \cite[Theorem~VI.1]{FlajoletSedgewick}) applied to \eqref{eq:A-local} and \eqref{eq:B-local} gives the stated asymptotics. Dividing them yields \eqref{eq:branch-ratio}.
\end{proof}

To identify the ratio $\beta_1/\alpha_1$, we differentiate at the Fricke fixed point.

\begin{lemma}\label{lem:derivative-ratio}
One has
\begin{equation}\label{eq:derivative-ratio}
\frac{\beta_1}{\alpha_1}
=\frac{\dfrac{\dd}{\dd\tau}\mathcal B(t(\tau))\big|_{\tau=\tau_*}}{Y'(\tau_*)}.
\end{equation}
\end{lemma}

\begin{proof}
By Proposition~\ref{prop:local-exponents}, $t(\tau)-t_0=c(\tau-\tau_*)^2+O((\tau-\tau_*)^3)$ with $c\neq 0$, so
\[
\left(1-\frac{t(\tau)}{t_0}\right)^{1/2}=\lambda(\tau-\tau_*)+O\bigl((\tau-\tau_*)^2\bigr)
\qquad(\lambda\neq 0).
\]
Substituting this into \eqref{eq:A-local} and \eqref{eq:B-local} gives
\[
Y(\tau)=\alpha_0+\alpha_1\lambda(\tau-\tau_*)+O\bigl((\tau-\tau_*)^2\bigr),
\]
\[
\mathcal B(t(\tau))=\beta_0+\beta_1\lambda(\tau-\tau_*)+O\bigl((\tau-\tau_*)^2\bigr).
\]
Differentiating at $\tau=\tau_*$ yields
\[
Y'(\tau_*)=\alpha_1\lambda,
\qquad
\frac{\dd}{\dd\tau}\mathcal B(t(\tau))\Big|_{\tau=\tau_*}=\beta_1\lambda,
\]
which is exactly \eqref{eq:derivative-ratio}.
\end{proof}

The Fricke period polynomial gives the final derivative identity.

\begin{proposition}\label{prop:derivative-identities}
At the fixed point $\tau_*$ one has
\begin{equation}\label{eq:E-derivative}
E(\tau_*)+\frac{E'(\tau_*)}{2\ii\sqrt2}=\frac{7}{32}\zeta(3),
\end{equation}
\begin{equation}\label{eq:Y-derivative}
Y'(\tau_*)=2\ii\sqrt2\,Y(\tau_*).
\end{equation}
\end{proposition}

\begin{proof}
Differentiate \eqref{eq:E-transform}. Since $W_8(\tau_*)=\tau_*$, $W_8'(\tau_*)=-1$, and $8\tau_*^2=-1$, the derivative of $(E|_{-2}W_8)(\tau)=8\tau^2E(W_8\tau)$ at $\tau_*$ is
\[
16\tau_*E(\tau_*)-8\tau_*^2E'(\tau_*)
=4\ii\sqrt2\,E(\tau_*)+E'(\tau_*).
\]
Adding the derivative of $E(\tau)$ and differentiating the right-hand side of \eqref{eq:E-transform} gives
\[
4\ii\sqrt2\,E(\tau_*)+2E'(\tau_*)=16\tau_*\cdot\frac{7}{32}\zeta(3)=\frac{7\ii\sqrt2}{8}\zeta(3),
\]
which is equivalent to \eqref{eq:E-derivative}.

Now differentiate \eqref{eq:Y-fricke}. Evaluating at $\tau_*$ gives
\[
-Y'(\tau_*)=-16\tau_*Y(\tau_*)-8\tau_*^2Y'(\tau_*).
\]
Since $-8\tau_*^2=1$, this becomes
\[
2Y'(\tau_*)=16\tau_*Y(\tau_*)=4\ii\sqrt2\,Y(\tau_*),
\]
which is exactly \eqref{eq:Y-derivative}.
\end{proof}

\begin{proof}[Proof of Theorem~\ref{thm:main-limit}]
From \eqref{eq:B-param} we have
\[
\mathcal B(t(\tau))=Y(\tau)E(\tau).
\]
Differentiating at $\tau=\tau_*$ and using \eqref{eq:Y-derivative}, we obtain
\begin{align*}
\frac{\dfrac{\dd}{\dd\tau}\mathcal B(t(\tau))\big|_{\tau=\tau_*}}{Y'(\tau_*)}
&=\frac{Y'(\tau_*)E(\tau_*)+Y(\tau_*)E'(\tau_*)}{Y'(\tau_*)}\\
&=E(\tau_*)+\frac{E'(\tau_*)}{2\ii\sqrt2}
=\frac{7}{32}\zeta(3)
\end{align*}
by \eqref{eq:E-derivative}. Therefore Lemma~\ref{lem:derivative-ratio} gives
\[
\frac{\beta_1}{\alpha_1}=\frac{7}{32}\zeta(3).
\]
Finally Lemma~\ref{lem:transfer} yields
\[
\lim_{n\to\infty}\frac{B_n^{(8)}}{s_n}=\frac{7}{32}\zeta(3).
\]
This is exactly \eqref{eq:main-limit}.
\end{proof}

\section{The continued-fraction normalization}\label{sec:cf}

We now give the elementary continuant bridge from the Apéry limit to the Batut--Olivier/Ramanujan-Machine normalization. For the continued fraction \eqref{eq:intro-Z1}, define
\[
a_0:=1,
\qquad
a_n:=(2n+1)(3n^2+3n+1)\quad(n\ge 1),
\qquad
b_n:=-n^6.
\]
Let $P_n,Q_n$ be the continuants determined by
\[
P_{-1}=1,\quad P_0=1,
\qquad
Q_{-1}=0,\quad Q_0=1,
\]
\begin{equation}\label{eq:continuant-rec}
U_n=a_nU_{n-1}+b_nU_{n-2}
\qquad(n\ge 1),
\end{equation}
where $U_n$ stands for either $P_n$ or $Q_n$. Then $P_n/Q_n$ is the $n$th convergent of \eqref{eq:intro-Z1}.

\begin{proposition}\label{prop:continuants}
For every $n\ge 0$ one has
\begin{equation}\label{eq:continuant-closed}
P_n=\frac{(n+1)!^3}{4^{n+1}}s_{n+1},
\qquad
Q_n=\frac{(n+1)!^3}{4^n}B_{n+1}^{(8)}.
\end{equation}
Consequently,
\begin{equation}\label{eq:convergent-ratio}
\frac{P_n}{Q_n}=\frac{s_{n+1}}{4B_{n+1}^{(8)}},
\qquad
\lim_{n\to\infty}\frac{P_n}{Q_n}=\frac{8}{7\zeta(3)}.
\end{equation}
\end{proposition}

\begin{proof}
Set
\[
\widetilde P_n:=\frac{(n+1)!^3}{4^{n+1}}s_{n+1}.
\]
Using \eqref{eq:s-rec} with index $n$, we obtain
\begin{align*}
\widetilde P_n
&=\frac{(n+1)!^3}{4^{n+1}}\cdot\frac{(2n+1)(12n^2+12n+4)s_n-16n^3s_{n-1}}{(n+1)^3}\\
&=(2n+1)(3n^2+3n+1)\frac{n!^3}{4^n}s_n-n^6\frac{(n-1)!^3}{4^{n-1}}s_{n-1}\\
&=a_n\widetilde P_{n-1}+b_n\widetilde P_{n-2}.
\end{align*}
Also $\widetilde P_0=1$ and $\widetilde P_1=20=a_1+b_1P_{-1}$, so $\widetilde P_n=P_n$ for all $n$.

The same calculation with $B_n^{(8)}$ in place of $s_n$ shows that
\[
\widetilde Q_n:=\frac{(n+1)!^3}{4^n}B_{n+1}^{(8)}
\]
satisfies the same recurrence \eqref{eq:continuant-rec}. Since $\widetilde Q_0=1$ and $\widetilde Q_1=21=a_1$, we have $\widetilde Q_n=Q_n$. This proves \eqref{eq:continuant-closed}. The limit in \eqref{eq:convergent-ratio} follows immediately from Theorem~\ref{thm:main-limit}.
\end{proof}

\begin{proof}[Proof of Corollary~\ref{cor:main-cf}]
By construction, $P_n/Q_n$ is the $n$th convergent of the continued fraction \eqref{eq:intro-Z1}. Proposition~\ref{prop:continuants} shows that these convergents tend to $8/(7\zeta(3))$. Hence
\[
\PCF\bigl((2n+1)(3n^2+3n+1),-n^6\bigr)=\frac{8}{7\zeta(3)}.
\]
\end{proof}

\section{Concluding remarks}\label{sec:conclusion}

The theorem-level identities treated here have earlier proof routes: the
level-$8$ Apéry constant through Golyshev's Beukers--Atkin--Lehner method and
through the Golyshev--Kerr--Sasaki motivic/normal-function framework, and
the continued fraction through Batut--Olivier. The expository contribution
of the present derivation is the explicit eta-product/Eichler-integral
computation with the Wronskian identity reducing the inhomogeneous companion
to a concrete weight-$4$ Eisenstein combination, the Mellin--Barnes residue
calculation giving the explicit Fricke period polynomial, and the elementary
continuant calculation identifying the resulting Apéry limit with the
Batut--Olivier/Ramanujan-Machine continued fraction.

\end{document}